\documentclass[12pt]{article}
\usepackage{graphicx}
\usepackage[ruled,vlined]{algorithm2e}
\usepackage{appendix}
\usepackage{amsmath}
\usepackage{amssymb}
\usepackage{amsthm}
\usepackage{multirow}
\usepackage{color}
\usepackage{longtable}
\usepackage{array}
\usepackage{url}
\usepackage{booktabs}
\usepackage{float}
\usepackage{mathtools}
\usepackage{tikz}
\usepackage{caption}

\newtheorem{theorem}{Theorem}[section]
\newtheorem{definition}[theorem]{Definition}
\newtheorem{lemma}[theorem]{Lemma}

\newtheorem{corollary}[theorem]{Corollary}

\theoremstyle{definition}

\title{Some tight lower bounds for Tur\'{a}n problems via constructions of multi-hypergraphs}
\author{Zixiang Xu$^{\text{a}}$, Tao Zhang$^{\text{b,}}$\thanks{Research was supported by the National Natural Science Foundation of China under Grant No. 11801109.}~ and  Gennian Ge$^{\text{a,}}$\thanks{Corresponding author (e-mail: gnge@zju.edu.cn). Research was supported by the National Natural Science Foundation of China under Grant No. 11971325 and Beijing Scholars
Program.}\\
\footnotesize $^{\text{a}}$ School of Mathematics Sciences, Capital Normal University, Beijing 100048, China.\\
\footnotesize $^{\text{b}}$ School of Mathematics and Information Science, Guangzhou University, Guangzhou 510006, China.\\}

\begin{document}

\date{}

\maketitle

\begin{abstract}
Recently, several hypergraph Tur\'{a}n problems were solved by the powerful random algebraic method. However, the random algebraic method usually requires some parameters to be very large, hence we are concerned about how these Tur\'{a}n numbers depend on such large parameters of the forbidden hypergraphs. In this paper, we determine the dependence on such specified large constant for several hypergraph Tur\'{a}n problems. More specifically, for complete $r$-partite $r$-uniform hypergraphs, we show that if $s_{r}$ is sufficiently larger than $s_{1},s_{2},\ldots,s_{r-1},$ then
\begin{equation*}
  \textup{ex}_{r}(n,K_{s_{1},s_{2},\ldots,s_{r}}^{(r)})=\Theta(s_{r}^{\frac{1}{s_{1}s_{2}\cdots s_{r-1}}}n^{r-\frac{1}{s_{1}s_{2}\cdots s_{r-1}}}).
\end{equation*}
For complete bipartite $r$-uniform hypergraphs, we prove that if $s$ is sufficiently larger than $t,$ we have
\begin{equation*}
  \textup{ex}_{r}(n,K_{s,t}^{(r)})=\Theta(s^{\frac{1}{t}}n^{r-\frac{1}{t}}).
\end{equation*}
In particular, our results imply that the famous K\H{o}v\'{a}ri--S\'{o}s--Tur\'{a}n's upper bound $\textup{ex}(n,K_{s,t})=O(t^{\frac{1}{s}}n^{2-\frac{1}{s}})$ has the correct dependence on large $t$. The main approach is to construct random multi-hypergraph via a variant of random algebraic method.


\smallskip

\end{abstract}

\section{Introduction}
The study of Tur\'{a}n problems is one of the essential ingredients in extremal graph theory. In $1907,$ Mantel \cite{1907Mantel} first showed that every $n$-vertex graph with more than $\frac{n^{2}}{4}$ edges contains a triangle. Later Tur\'{a}n \cite{1941Turan} generalized this result to $K_{\ell}$-free graphs for $\ell\geqslant 4$. For general graph $H,$ Erd\H{o}s and Stone \cite{1946ErodsBAMS} gave the asymptotic results for the Tur\'{a}n number $\textup{ex}(n,H)$. However, to determine the exact asymptotic results for $\textup{ex}(n,H)$ is challenging when $H$ is a bipartite graph. Complete bipartite graphs and even cycles are two important objects when we study such degenerate Tur\'{a}n problems. For complete bipartite graphs, the result of K\H{o}v\'{a}ri, S\'{o}s and Tur\'{a}n \cite{Kovari1954} showed that $\textup{ex}(n,K_{s,t})=O(n^{2-\frac{1}{s}})$ for any integers $t>s$. When $s=2,3,$ Erd\H{o}s, R\'{e}nyi and S\'{o}s \cite{Erdos1966} and Brown \cite{Brown1966} gave the matched lower bounds respectively. For general values of $s$ and $t$, Koll\'{a}r, R\'{o}nyai and Szab\'{o} \cite{KRS96} indicated that $\text{ex}(n,K_{s,t})=\Omega(n^{2-\frac{1}{s}})$ when $t\geqslant s!+1$ via norm graphs, and then this condition was improved to $t\geqslant (s-1)!+1$ by Alon, R\'{o}nyai and Szab\'{o} \cite{ARS99}. Recently, Blagojevi\'{c}, Bukh and Karasev \cite{TopologicalKST2013} gave a new type of $K_{s,t}$-free graph via topological obstructions and algebraic constructions. Then Bukh \cite{Bukh2015} established an elegant method, which is named random algebraic construction, to show that $\textup{ex}(n,K_{s,t})=\Omega(n^{2-\frac{1}{s}})$ when $t$ is sufficiently larger than $s$. From then on, the random algebraic method was applied to several Tur\'{a}n type problems, see \cite{Bukh2018,BukhTailtheta2018,Conlon2014,Ma2018,XuZhangGe2019}.

For even cycles, the extremal results of $\textup{ex}(n,C_{2\ell})$ were first studied by Erd\H{o}s \cite{1938erdos}, and then Bondy and Simonovits \cite{BondyEvenCycle1974} gave a general upper bound $\textup{ex}(n,C_{2\ell})\leqslant 100\ell n^{1+\frac{1}{\ell}}.$ Recently, Bukh and Jiang \cite{ZilinCPC2017} improved the upper bound to $\textup{ex}(n,C_{2\ell})\leqslant 80\sqrt{\ell}\log{\ell} n^{1+\frac{1}{\ell}}$, and this upper bound is the current record. However, the order of magnitude for $\textup{ex}(n,C_{2\ell})$ is unknown for any $\ell\notin\{2,3,5\},$ see \cite{Brown1966,Erdos1966,WengerC4C6C10}. For general $\ell\notin\{2,3,5,7\}$, the best known lower bounds for $\textup{ex}(n,C_{2\ell})$ were obtained by Lazebnik, Ustimenko and Woldar \cite{1995LowerBoundC2l}, and the best known lower bound for $\textup{ex}(n,C_{14})$ was shown in \cite{2012EX14}.

Due to the similarity of theta graphs and even cycles, the Tur\'{a}n number of theta graphs have been studied recently. Let theta graph $\Theta_{\ell,t}$ be a graph made of $t$ internally disjoint paths of length $\ell$ connecting two endpoints. Since it is unclear whether $\textup{ex}(n,C_{2\ell})=\Omega(n^{1+\frac{1}{\ell}})$ holds in general, the study of $\textup{ex}(n,\Theta_{\ell,t})$ is of interest. Faudree and Simonovits \cite{Faudree1983} first showed the general upper bound $\textup{ex}(n,\Theta_{\ell,t})=O_{\ell,t}(n^{1+\frac{1}{\ell}})$. Recently, Conlon \cite{Conlon2014} showed the matched lower bounds when $t$ is a sufficiently large constant. After that Bukh and Tait \cite{BukhTailtheta2018} studied the behavior of $\textup{ex}(n,\Theta_{\ell,t})$ when $\ell$ is fixed and $t$ is very large, and they further determined the dependence on $t$ when $\ell$ is odd. When $\ell$ and $t$ are relatively small, Verstra\"{e}te and Williford \cite{2019theta43} showed that $\textup{ex}(n,\Theta_{4,3})\geqslant (\frac{1}{2}-o(1))n^{\frac{5}{4}},$ and this result is perhaps the evidence that the Tur\'{a}n number of the octagon is also of order $n^{\frac{5}{4}}.$

On the contrary to the simple graph cases, there are only a few results on hypergraph Tur\'{a}n problems. There are two ways of generalizing complete bipartite graphs to hypergraphs. The first one is complete $r$-partite $r$-uniform hypergraph $K_{s_{1},s_{2},\ldots,s_{r}}^{(r)}.$ Mubayi \cite{Mubayi2002} conjectured that $\text{ex}_{r}(n,K_{s_{1},s_{2},\ldots,s_{r}}^{(r)})=\Theta(n^{r-\frac{1}{s_{1}s_{2}\cdots s_{r-1}}})$, where $s_{1}\leqslant s_{2}\leqslant\dots\leqslant s_{r},$ and he proved this conjecture in certain situations. Recently, Ma, Yuan and Zhang \cite{Ma2018} showed that if $s_{r}$ is sufficiently larger than $s_{1},s_{2},\ldots,s_{r-1},$ then this conjecture is true. The other object is complete bipartite $r$-uniform hypergraph $K_{s,t}^{(r)}$. In \cite{Mubayi2004}, Mubayi and Verstra\"{e}te showed some general bounds for $\textup{ex}_{r}(n,K_{s,t}^{(r)})$ when $s<t$. More recently, Xu, Zhang and Ge \cite{XuZhangGe2019} gave the lower bound $\textup{ex}_{r}(n,K_{s,t}^{(r)})=\Omega(n^{r-\frac{1}{t}})$ when $s$ is sufficiently larger than $t$, and the general upper bound $\textup{ex}_{r}(n,K_{s,t}^{(r)})=O(s^{\frac{1}{t}}n^{r-\frac{1}{t}})$ when $s$ is large.

The Tur\'{a}n problem for cycles in hypergraphs has been investigated for so-called Berge cycles. Gy\H{o}ri \cite{Gyori2006CPC} first determined $\textup{ex}_{3}(n,C_{3}^{B})$ for all $n$, then Bollob\'{a}s and Gy\H{o}ri \cite{Bollobas2008DM} showed that $\textup{ex}_{3}(n,C_{5}^{B})=O(n^{\frac{3}{2}}).$ Gy\H{o}ri and Lemons \cite{Lemons2012CPC} showed the general upper bounds $\textup{ex}_{r}(n,C_{2\ell}^{B})=O(n^{1+\frac{1}{\ell}})$ and $\textup{ex}_{r}(n,C_{2\ell+1}^{B})=O(n^{1+\frac{1}{\ell}})$ for all $\ell\geqslant 2$ and $r\geqslant 3.$ It is known in \cite{2017ENDM} that $\textup{ex}_{r}(n,C_{4}^{B})=\Theta(n^{\frac{3}{2}})$ when $2\leqslant r\leqslant 6$, but the order of magnitude is still unknown for $r\geqslant 7$. It is widely open whether the general upper bounds are tight for all $r,\ell\geqslant 3$. For more extremal results of Berge cycles, we refer the readers to \cite{GerbnerSIAM2017,JiangJCTB2018,Jacques2016} and the references therein.

Since there are few exact asymptotic results of $\textup{ex}_{r}(n,C_{2\ell}^{B})$, we are interested in the generalization of theta graphs to hypergraphs. Let $r$-uniform Berge theta hypergraph $\Theta_{\ell,t}^{B}$ be a set of distinct vertices $x,y,v_{1}^{1},\cdots,v_{\ell-1}^{1},\cdots,v_{1}^{t},$ $\cdots,v_{\ell-1}^{t}$ and a set of distinct edges $e_{1}^{1},\cdots,$ $e_{\ell}^{1},\cdots,e_{1}^{t},\cdots,e_{\ell}^{t}$ such that $\{x,v_{1}^{i}\}\subset e_{1}^{i},$ $\{v_{j-1}^{i},v_{j}^{i}\}\subset e_{j}^{i}$ and $\{v_{\ell-1}^{i},y\}\subset e_{\ell}^{i}$ for $1\leqslant i\leqslant t$ and $2\leqslant j\leqslant \ell-1.$ Recently, He and Tait \cite{TaitBergetheta2018} studied the Tur\'{a}n number $\textup{ex}_{r}(n,\Theta_{\ell,t}^{B}),$ in particular, they showed that for fixed $\ell$ and $r$, there is a large constant $t$ such that $\textup{ex}_{r}(n,\Theta_{\ell,t}^{B})$ can be determined in order of magnitude.

 As far as we know, the random algebraic method always requires one of some parameters to be very large, therefore determining the dependence on this large parameter is interesting. Inspired by Bukh-Tait's results on theta graph \cite{BukhTailtheta2018}, we investigate three important objects including complete $r$-partite $r$-uniform hypergraphs, complete bipartite $r$-uniform hypergraphs and Berge theta hypergraphs. Our main idea is to construct the random multi-hypergraphs via a variant of random algebraic method, and our main contributions in this paper are listed as follows.

  \begin{itemize}
    \item \textbf{Complete $r$-partite $r$-uniform hypergraphs:}
    \begin{theorem}\label{Thm:Degenerate}
   For any positive integers $s_{1},s_{2},\ldots,s_{r-1}$ and $r\geqslant 2,$ when $s_{r}$ is sufficiently large, we have
    \begin{equation*}
      \textup{ex}_{r}(n,K_{s_{1},s_{2},\ldots,s_{r}}^{(r)})=\Omega(s_{r}^{\frac{1}{s_{1}s_{2}\ldots s_{r-1}}}n^{r-\frac{1}{s_{1}s_{2}\ldots s_{r-1}}}).
    \end{equation*}
\end{theorem}
By the result of \cite[Lemma 3.1]{Ma2018}, the dependence on large $s_{r}$ is tight.

    \item \textbf{Complete bipartite $r$-uniform hypergraphs:}

   In $2004$, Mubayi and Verstra\"{e}te \cite{Mubayi2004} considered a hypergraph extension of the complete bipartite graph. In this paper, we call it complete bipartite $r$-uniform hypergraph for simplicity. Recall the definition of complete bipartite $r$-uniform hypergraph as follows.
\begin{definition}[Complete bipartite $r$-uniform hypergraph]
  Let $X_{1},X_{2},\ldots,X_{t}$ be $t$ pairwise disjoint sets of size $r-1$, and let $Y$ be a set of $s$ elements, disjoint from $\bigcup\limits_{i\in [t]}X_{i}$. Then $K_{s,t}^{(r)}$ denotes the complete bipartite $r$-uniform hypergraph with vertex set $(\bigcup\limits_{i\in [t]}X_{i})\cup Y$ and edge set $\{X_{i}\cup\{y\}:i\in [t],y\in Y\}$.
\end{definition}
Using our tools, we obtain the following result.
    \begin{theorem}\label{Thm:bipartite}
 For any positive integers $t$ and $r\geqslant 2,$ when $s$ is sufficiently large, we have
    \begin{equation*}
      \textup{ex}_{r}(n,K_{s,t}^{(r)})=\Omega(s^{\frac{1}{t}}n^{r-\frac{1}{t}}).
    \end{equation*}
\end{theorem}
By the result of \cite[Theorem 1.3]{XuZhangGe2019}, the dependence on large $s$ is tight.

As a corollary, both of Theorems \ref{Thm:Degenerate} and \ref{Thm:bipartite} imply that the upper bound of K\H{o}v\'{a}ri, S\'{o}s and Tur\'{a}n \cite{Kovari1954} is tight for all sufficiently large $t$.

\begin{corollary}
  For given positive integer $s,$ when $t$ is sufficiently large, we have
  \begin{equation*}
    \textup{ex}(n,K_{s,t})=\Theta(t^{\frac{1}{s}}n^{2-\frac{1}{s}}).
  \end{equation*}
\end{corollary}

    \item \textbf{Berge theta hypergraphs:}

     For general $\ell,$ by the upper bound for $\textup{ex}(n,\Theta_{\ell,t})=O_{\ell}(t^{1-\frac{1}{\ell}}n^{1+\frac{1}{\ell}})$ in \cite{BukhTailtheta2018}, Gerbner, Methuku and Palmer~\cite{UpperBoundLemma2018} showed the following upper bound when $t$ is sufficiently large.
  \begin{theorem}[\cite{UpperBoundLemma2018}]\label{Thm:Bergeupperbound}
    For fixed $\ell\geqslant 2,$ when $t$ is sufficiently large, we have
    \begin{equation*}
      \textup{ex}_{r}(n,\Theta_{\ell,t}^{B})=O_{\ell,r}(t^{r-1-\frac{1}{\ell}}n^{1+\frac{1}{\ell}}).
    \end{equation*}
  \end{theorem}
    However we do not know whether the general upper bound is tight. Using our tools, we can show a lower bound as follows.

    \begin{theorem}\label{Thm:Bergelowerbound}
    Let $\ell\geqslant 2$ be a fixed integer, when $t$ is sufficiently large, we have
    \begin{equation*}
      \textup{ex}_{r}(n,\Theta_{\ell,t}^{B})=\Omega_{\ell,r}(t^{\frac{1}{\ell}}n^{1+\frac{1}{\ell}}).
    \end{equation*}
  \end{theorem}
  \end{itemize}

  The rest of this paper is organized as follows.~In Section \ref{Sec:Preliminary}, we introduce some basic facts about random algebraic method. In Section \ref{Sec:ProofMain}, we use a variant of random algebraic method to construct various multi-hypergraphs and then prove our main results of complete $r$-partite $r$-uniform hypergraphs, complete bipartite $r$-uniform hypergraphs and Berge theta hypergraphs. Finally we conclude in Section \ref{Sec:remarks}, and provide some remarks and open problems on the main topics.

\section{Preliminaries of random algebraic method}\label{Sec:Preliminary}
 Let $t,r$ be positive integers with $r\geqslant 2,$ $q$ be a sufficiently large prime power, and $\mathbb{F}_{q}$ be the finite field of order $q$. Let $\textbf{X}^{i}=(X_{1}^{i},X_{2}^{i},\ldots,X_{t}^{i})\in \mathbb{F}_{q}^{t}$ for each $i\in [r]$. Consider polynomials $f\in \mathbb{F}_{q}[\textbf{X}^{1},\textbf{X}^{2},\ldots,\textbf{X}^{r}]$ with $rt$ variables over $\mathbb{F}_{q}$. We say such a polynomial $f$ has degree at most $d$ in $\textbf{X}^{i},$ if each of its monomials has degree at most $d$ with respect to $\textbf{X}^{i}$, that is,  $(X_{1}^{i})^{\alpha_{1}}(X_{2}^{i})^{\alpha_{2}}\cdots(X_{t}^{i})^{\alpha_{t}}$ satisfies $\sum\limits_{j=1}^{t}\alpha_{j}\leqslant d.$ Moreover, a polynomial $f$ is called symmetric if exchanging $\textbf{X}^{i}$ with $\textbf{X}^{j}$ for every $1\leqslant i\leqslant j\leqslant r$ does not affect the value of $f$. Let $\mathcal{P}_{d}\subseteq \mathbb{F}_{q}[\textbf{X}^{1},\textbf{X}^{2},\ldots,\textbf{X}^{r}]$ be the set of all symmetric polynomials of degree at most $d$ in $\textbf{X}^{i}$ for every $1\leqslant i \leqslant r.$

 We use the term $random\ polynomial$ to represent a polynomial chosen uniformly at random from $\mathcal{P}_{d}.$ Since the constant term of a random polynomial is chosen uniformly from $\mathbb{F}_{q},$ one can easily show that
 \begin{equation*}
   \mathbb{P}[f(v_{1},v_{2},\ldots,v_{r})=0]=\frac{1}{q}
 \end{equation*}
 for a random polynomial $f$ and any fixed $r$-tuple $(v_{1},v_{2},\ldots,v_{r}).$

In our constructions of random hypergraphs, the edges will appear when one polynomial or a system of polynomials vanishes, hence we can describe subhypergraphs as varieties. Let $\bar{\mathbb{F}}_{q}$ be the algebraic closure of $\mathbb{F}_{q},$ a variety over $\bar{\mathbb{F}}_{q}$ is a set of the form
 \begin{equation*}
   W=\{x\in \bar{\mathbb{F}}_{q}^{t}: f_{1}(x)=f_{2}(x)=\cdots=f_{s}(x)=0\}
 \end{equation*}
 for given polynomials $f_{1},f_{2},\ldots,f_{s}.$ That is, a variety is the set of common roots of a set of polynomials. Let $W(\mathbb{F}_{q})=W\cap \mathbb{F}_{q},$ and we say that $W$ has complexity at most $M$ if the above parameters $s,t$ and the maximum degree of the polynomials are all bounded by $M$.

 Now we introduce two important lemmas which will be useful in our constructions. The first lemma is the key insight of the random algebraic construction, which provides very non-smooth probability distributions. While the second lemma will help us calculate the probability in certain situations.

 \begin{lemma}[\cite{Bukh2018}]\label{Lem:variety}
   Suppose $W$ and $D$ are varieties over $\bar{\mathbb{F}}_{q}$ of complexity at most $M$ which are defined over $\mathbb{F}_{q}.$ Then either $|W(\mathbb{F}_{q})\setminus D(\mathbb{F}_{q})|\leqslant c_{M}$ or $|W(\mathbb{F}_{q})\setminus D(\mathbb{F}_{q})|\geqslant \frac{q}{2},$ where $c_{M}$ depends only on $M$.
 \end{lemma}

\begin{lemma}[\cite{Ma2018}]\label{Lem:probability}
    Given a set $U\subseteq \binom{\mathbb{F}_{q}^{t}}{r},$ let $V\subseteq \mathbb{F}_{q}^{t}$ be the set consisting of all points appeared as an element of an $r$-tuple in $U$. Suppose that $\binom{|U|}{2}<q,$ $\binom{|V|}{2}<q$ and $|U|\leqslant d.$ If $f$ is a random polynomial chosen from $\mathcal{P}_{d}$, then
  \begin{equation*}
    \mathbb{P}[f(u^{1},u^{2},\ldots,u^{r})=0,  \forall \{u^{1},u^{2},\ldots,u^{r}\}\in U]=q^{-|U|}.
  \end{equation*}
 \end{lemma}

\section{Constructions of random multi-hypergraphs }\label{Sec:ProofMain}
In this section, we will show some lower bounds for $\textup{ex}_{r}(n,\mathcal{T})$ via constructions of random multi-hypergraphs. Here we illustrate our main idea briefly. We first construct a random multi-hypergraph by taking union of $h$ random hypergraphs. Our goal is to show that averagely this multi-hypergraph contains many edges, with very few copies of $\mathcal{T}$ and multiple edges. For different forbidden hypergraphs $\mathcal{T}$, we will define the corresponding bad structures and estimate their number. Finally we will delete one vertex from each bad structure and delete all of the multiple edges to obtain a new hypergraph, which is $\mathcal{T}$-free and has expected number of edges.

There are three major ingredients. First, the random multi-hypergraph is the union of several random hypergraphs, which are defined by some bounded-degree random polynomials. Hence Lemma~\ref{Lem:probability} can help us estimate the expectation of number of single edges and multiple edges, respectively. Second, since the independence between different random hypergraphs, Lemma~\ref{Lem:probability} still works when we need to estimate the expectation of number of some structures, though the edges of such structures in multi-hypergraph are from distinct original random hypergraphs. The third ingredient is that, since the random hypergraphs are defined by bounded-degree polynomials, we can define the bad structure in multi-hypergraph properly. Then we regard the set of bad structures as variety, Lemma~\ref{Lem:variety} can help us bound the expectation of number of the bad structures, combining the Markov's inequality.

\subsection{Complete $r$-partite $r$-uniform hypergraphs}\label{Sec:Complete degenerate hypergraph}
In this subsection we consider the Tur\'{a}n number of complete $r$-partite $r$-uniform hypergraph $K_{s_{1},s_{2},\ldots,s_{r}}^{(r)}.$ We construct the random multi-hypergraph based on the construction of \cite{Ma2018}.

\begin{definition}[\cite{Ma2018}]\label{Def:DegenerateRandomGraph}
 For given integers $s_{1},s_{2},\ldots,s_{r-1}$ and $r,$ let $b=\prod\limits_{i=1}^{r-1}s_{i},$ $t=\sum\limits_{i=1}^{r-1}s_{i},$ $s=b(t-1)+2$ and $d=bs.$ Let $N=q^{b},$ we pick a symmetric polynomial $f$ from $\mathcal{P}_{d}$ uniformly at random. Then we define an $r$-uniform hypergraph $\mathcal{G}$ on $N$ vertices as following: the vertex set is a copy of $\mathbb{F}_{q}^{b},$ and the $r$-tuple $\{v_{1},v_{2},\ldots,v_{r}\}$ forms an edge of $\mathcal{G}$ if and only if $f(v_{1},v_{2},\ldots,v_{r})=0.$
\end{definition}

We pick $h$ independent random symmetric polynomials $f_{1},f_{2},\ldots,f_{h}$ from $\mathcal{P}_{d}$ uniformly, and denote their associated hypergraphs as $\mathcal{G}_{1},\mathcal{G}_{2},\ldots,\mathcal{G}_{h}.$ Let $\bar{\mathcal{G}}$ be a multi-hypergraph which is the union of $\mathcal{G}_{1},\mathcal{G}_{2},\ldots,\mathcal{G}_{h}.$ In the multi-hypergraph $\bar{\mathcal{G}}$, let $T$ be a fixed labelled copy of $K_{s_{1},s_{2},\ldots,s_{r-1},1}^{(r)}$ and denote its vertices as $u$ and $v_{j}^{i}$ for $1\leqslant i\leqslant r-1,$ $1\leqslant j\leqslant s_{i}$ such that $v_{1}^{i},v_{2}^{i},\ldots,v_{s_{i}}^{i}$ are in the same specified part. Fix a sequence of vertices $w_{1}^{i},w_{2}^{i},\ldots,w_{s_{i}}^{i}$ for $1\leqslant i\leqslant r-1,$ and these vertices form $b$ distinct $(r-1)$-tuples according to the fixed labelled copy of $T$. Since every edge can be in one of $h$ distinct original hypergraphs $\mathcal{G}_{k},$ there are totally $h^{b}$ types of given labelled $T.$ Let $p$ be a positive integer and $W$ be the family of copies of $T$ which contains the fixed sequence $w_{1}^{i},w_{2}^{i},\ldots,w_{s_{i}}^{i}$ in $\bar{\mathcal{G}}$ for $1\leqslant i\leqslant r-1.$ We call a sequence of vertices $w_{1}^{i},w_{2}^{i},\ldots,w_{s_{i}}^{i}$ a $p$-bad sequence, if the corresponding set $W$ has size $|W|\geqslant p.$ Let $B_{p}$ be the set of all $p$-bad sequences in $\bar{\mathcal{G}}$.

\begin{lemma}\label{Lem:DegenerateBadSet}
   There exist constants $p$ and $C$ depending on $s_{1},s_{2},\ldots, s_{r-1},r$ such that
   \begin{equation*}
     \mathbb{E}[|B_{ph^{b}}|]\leqslant Ch^{b}N^{1-\frac{2}{b}}.
   \end{equation*}
\end{lemma}
\begin{proof}
  Fix a type $I\in[h]^{b},$ call a sequence of vertices $\{w_{1}^{i},w_{2}^{i},\ldots,w_{s_{i}}^{i} : 1\leqslant i\leqslant r-1\}$ a $(p,I)$-bad sequence if the corresponding set $W_{I}$ has size $|W_{I}|\geqslant p,$ where $p$ will be determined later. By the linearity of expectation, it suffices to prove that the expected number of $(p,I)$-bad sequences is $O(N^{1-\frac{2}{b}})$ since the total number of types is $h^{b}.$

  Now we focus on the size of $W_{I}$. It is difficult to estimate $|W_{I}|$ directly, hence we consider the $s$-th moment of $|W_{I}|.$ Note that $|W_{I}|^{s}$ counts the number of ordered collections of $s$ copies of $T$ from $W_{I}$, and these copies of $T$ may be the same, hence each member of such collections can be an element $P$ in
\begin{equation*}
  \mathcal{K}:=\{K_{s_{1},s_{2},\ldots,s_{r-1},1}^{(r)},K_{s_{1},s_{2},\ldots,s_{r-1},2}^{(r)},\ldots,K_{s_{1},s_{2},\ldots,s_{r-1},s}^{(r)}\}.
\end{equation*}
 For given $P\in \mathcal{K},$ let $N_{s}(P)$ be the number of all possible ordered collections of $s$ copies of $T\in W_{I}$ which appear in $\bar{\mathcal{G}}$ as a copy of $P$. Note that the number of unfixed vertices in $P$ is $|P|-t,$ so $N_{s}(P)=O(n^{|P|-t}).$ The edge set of $P$ can be written as
 \begin{equation*}
   E(P)=E(P_{1})\cup E(P_{2})\cup\cdots\cup E(P_{h})
 \end{equation*}
 according to the fixed type $I,$ where $E(P_{i})$ consists of edges from the original hypergraph $\mathcal{G}_{i},$ $i=1,2,\ldots,h.$ Since the random hypergraphs $\mathcal{G}_{1},\mathcal{G}_{2},\ldots,\mathcal{G}_{h}$ are picked independently, by Lemma~\ref{Lem:probability}, we have
 \begin{equation*}
   \mathbb{E}[|W|^{s}]=\sum\limits_{P\in\mathcal{K}}N_{s}(P)\prod\limits_{i=1}^{h}q^{e(P_{i})}=\sum\limits_{P\in\mathcal{K}}O(N^{|P|-t})\cdot q^{b(|P|-t)}=O(1).
 \end{equation*}

For a fixed type $I$, $W_{I}$ is a variety which consists of vertices $x\in \mathbb{F}_{q}^{b}$ satisfying the system of $b$ equations $f_{k}(w_{j_{1}}^{1},w_{j_{2}}^{2},\ldots,w_{j_{r-1}}^{r-1},x)=0$ for all $1\leqslant i\leqslant r-1$ and $1\leqslant j_{i}\leqslant s_{i},$ where the choice of $k$ for certain edge $(w_{j_{1}}^{1},w_{j_{2}}^{2},\ldots,w_{j_{r-1}}^{r-1},x)$ only depends on the fixed type $I$. Note that every random polynomial $f_{k}(w_{j_{1}}^{1},w_{j_{2}}^{2},\ldots,w_{j_{r-1}}^{r-1},x)$ is chosen from $\mathcal{P}_{d},$ we have that $W_{I}$ has complexity at most $bs.$ By Lemma \ref{Lem:variety}, either $|W_{I}|\leqslant c_{I}$ or $|W_{I}|\geqslant \frac{q}{2}.$ Then we can use the Markov's inequality to bound the probability as
\begin{equation*}
  \mathbb{P}[|W_{I}|> c_{I}]=\mathbb{P}[|W_{I}|\geqslant \frac{q}{2}]=\mathbb{P}[|W_{I}|^{s}\geqslant (\frac{q}{2})^{s}]\leqslant \frac{\mathbb{E}[|W_{I}|^{s}]}{(\frac{q}{2})^{s}}=\frac{O(1)}{q^{s}}.
\end{equation*}
Let $p=\max\limits_{I\in [h]^{b}}{c_{I}},$ the expected number of $(p,I)$-bad sequences is at most $t!N^{t}\cdot\frac{O(1)}{q^{s}}=O(N^{1-\frac{2}{b}}).$ By the linearity of expectation, Lemma \ref{Lem:DegenerateBadSet} follows since the total number of types is $h^{b}.$
\end{proof}

Now we are ready to prove our main result of complete $r$-partite $r$-uniform hypergraphs.
\begin{proof}[\textbf{Proof of Theorem \ref{Thm:Degenerate}}]
   Let $\bar{\mathcal{G}}$ be the multi-hypergraph defined as above. It is easy to see the expected number of edges in $\bar{\mathcal{G}}$ is $\frac{h}{q}\binom{N}{r}.$ Let $e_{M}$ be the number of multiple edges, we can bound the expected number of $e_{M}$ as
  \begin{equation*}
    \mathbb{E}[e_{M}]\leqslant \binom{N}{r}\sum\limits_{i=2}^{h}\binom{h}{i}q^{-i}=O(N^{r-\frac{2}{b}}).
  \end{equation*}
  Moreover, by Lemma \ref{Lem:DegenerateBadSet}, the expected number of $ph^{b}$-bad sequences is at most $Ch^{b}N^{1-\frac{2}{b}}.$ We remove all of the multiple edges and remove one vertex from each $ph^{b}$-bad sequence to obtain a new hypergraph $\mathcal{G}',$ since each vertex is contained in at most $O(N^{r-1})$ edges, hence the expected number of edges in $\mathcal{G}'$ is at least
  \begin{equation*}
    \frac{h}{q}\binom{N}{r}-\binom{N}{r}\sum\limits_{i=2}^{h}\binom{h}{i}q^{-i}-O(N^{r-1})Ch^{b}N^{1-\frac{2}{b}}.
  \end{equation*}
  When $s_{r}$ is sufficiently large, we choose $h=(\frac{s_{r}}{p})^{\frac{1}{b}},$ then there exists a $K_{s_{1},s_{2},\ldots,s_{r}}^{(r)}$-free hypergraph with $\Omega(s_{r}^{\frac{1}{s_{1}s_{2}\ldots s_{r-1}}}n^{r-\frac{1}{s_{1}s_{2}\ldots s_{r-1}}})$ edges, the proof of Theorem \ref{Thm:Degenerate} is finished.
\end{proof}

\subsection{Complete bipartite $r$-uniform hypergraphs}
In this subsection, we consider the Tur\'{a}n number of complete bipartite $r$-uniform hypergraph $K_{s,t}^{(r)}.$ We still take advantage of the construction in \cite{XuZhangGe2019}.

\begin{definition}[\cite{XuZhangGe2019}]
  For given integers $t$ and $r,$ let $N=q^{t},$ $m=(r-1)t^{2}-t+2,$ and $d=mt,$ we pick a symmetric polynomial $f$ from $\mathcal{P}_{d}$ uniformly at random. Then we define an $r$-uniform hypergraph $\mathcal{H}$ on $N$ vertices as following: the vertex set is a copy of $\mathbb{F}_{q}^{t},$ and the $r$-tuple $\{v_{1},v_{2},\ldots,v_{r}\}\in \binom{\mathbb{F}_{q}^{t}}{r}$ forms an edge of $\mathcal{H}$ if and only if $f(v_{1},v_{2},\ldots,v_{r})=0.$
\end{definition}
We then choose $h$ independent random symmetric polynomials $f_{1},f_{2},\ldots,f_{h}$ from $\mathcal{P}_{d}$ uniformly, and denote their associated hypergraphs as $\mathcal{H}_{1},\mathcal{H}_{2},\ldots,\mathcal{H}_{h}.$ Let $\bar{\mathcal{H}}$ be a multi-hypergraph which is the union of $\mathcal{H}_{1},\mathcal{H}_{2},\ldots,\mathcal{H}_{h}.$ In the multi-hypergraph $\bar{\mathcal{H}}$, let $R$ be a fixed labelled copy of $K_{1,t}^{(r)},$ and we denote its vertices as $a$ and $u_{j}^{i}$ for $1\leqslant j\leqslant t$ and $i\in [r-1]$ such that $u_{j}^{1},u_{j}^{2},\ldots,u_{j}^{r-1}$ form $t$ distinct $(r-1)$-tuples corresponding to the fixed labelled copy of $R$. Since each edge of $R$ can belong to one of $h$ distinct original hypergraphs $\mathcal{H}_{k},$ there are in total $h^{t}$ types of given labelled copy of $R.$ Now fix any sequence of vertices $w_{j}^{i}$ for $1\leqslant j\leqslant t$ and $i\in [r-1]$ in $\bar{\mathcal{H}}$.  Let $W$ be the family of copies of $R$ in $\bar{\mathcal{H}}$ such that $w_{j}^{i}$ corresponds to $u_{j}^{i}$ for all $1\leqslant j\leqslant t$ and $i\in [r-1]$. We say such a sequence $p$-bad if the corresponding set $W$ satisfies $|W|\geqslant p.$ Let $B_{p}$ be the set of all $p$-bad sequences in the multi-hypergraph $\bar{\mathcal{H}}.$

\begin{lemma}\label{Lem:BipartiteBadSet}
    There exist constants $p=p(t,r)$ and $C=C(t,r)$ such that
   \begin{equation*}
     \mathbb{E}[|B_{ph^{t}}|]\leqslant Ch^{t}q^{t-2}.
   \end{equation*}
\end{lemma}

\begin{proof}
Fix a type $J\in[h]^{t},$ call a sequence of vertices $\{w_{j}^{1},w_{j}^{2},\ldots,w_{j}^{r-1} : 1\leqslant j\leqslant t\}$ a $(p,J)$-bad sequence if the corresponding set $W_{J}$ has cardinality $|W_{J}|\geqslant p$ with $p$ to be determined later. In the following we will prove that the expected number of $(p,J)$-bad sequences is $O(q^{t-2}).$

 We prefer to bound the value of $|W_{J}|^{m}$ rather than estimate $|W_{J}|$ directly. Note that $|W_{J}|^{m}$ counts the number of ordered collections of $m$ copies of $R$ from $W_{J}$, where these copies of $R$ may be identical. So each member of such collections can be an element $L$ in
\begin{equation*}
  \mathcal{L}:=\{K_{1,t}^{(r)},K_{2,t}^{(r)},\ldots,K_{m,t}^{(r)}\}.
\end{equation*}
For given $L\in \mathcal{L}$, denote $N_{m}(L)$ as the total number of all possible ordered collections of $m$ copies of $R\in W_{J}$, which could appear in $\bar{\mathcal{H}}$ as a copy of $L$. Note that the number of unfixed vertices in $L$ is $|L|-t(r-1),$ so $N_{m}(L)=O(q^{t(|L|-t(r-1))}).$ On the other hand, according to the type $J\in[h]^{t},$ the edge set $E(L)$ can be written as
\begin{equation*}
  E(L)=E(L_{1})\cup E(L_{2})\cup\cdots\cup E(L_{h}),
\end{equation*}
where $E(L_{i})$ consists of edges from the original hypergraph $\mathcal{H}_{i},$ $i=1,2,\ldots,h.$ Since the random hypergraphs $\mathcal{H}_{1},\mathcal{H}_{2},\ldots,\mathcal{H}_{h}$ are picked independently, by Lemma~\ref{Lem:probability}, we have
\begin{equation*}
  \mathbb{E}[|W_{J}|^{m}]=\sum\limits_{L\in \mathcal{L}}N_{m}(L)\prod\limits_{i=1}^{h}q^{-e(L_{i})}=\sum\limits_{L\in \mathcal{L}}O(q^{t(|L|-t(r-1))})q^{-e(L)}=O(1).
\end{equation*}
Note that $W_{J}$ is a variety which consists of vertices $x\in \mathbb{F}_{q}^{t}$ satisfying the system of $t$ equations
\begin{equation*}
  f_{k}(w_{j}^{1},w_{j}^{2},\ldots,w_{j}^{r-1},x)=0
  \end{equation*}
  for $1\leqslant j\leqslant t.$ The choice of $k$ is dependent on the fixed type $J,$ and $f_{k}$ is the random polynomial used to define random hypergraph $\mathcal{H}_{k}.$ It is easy to check that for each $k\in [h],$ the random polynomial $f_{k}(w_{j}^{1},w_{j}^{2},\ldots,w_{j}^{r-1},x)$ has degree at most $d,$ hence the variety $W_{J}$ has complexity at most $d.$ Then by Lemma \ref{Lem:variety}, either $|W_{J}|\leqslant c_{J}$ or $|W_{J}|\geqslant \frac{q}{2},$ where $c_{J}$ is dependent on $d$ and the type $J.$ With the Markov's inequality, we obtain that
  \begin{equation*}
    \mathbb{P}[|W_{J}|> c_{J}]=\mathbb{P}[|W_{J}|\geqslant \frac{q}{2}]=\mathbb{P}[|W_{J}|^{m}\geqslant (\frac{q}{2})^{m}]\leqslant \frac{\mathbb{E}[|W_{J}|^{m}]}{(\frac{q}{2})^{m}}=\frac{O(1)}{q^{m}}.
  \end{equation*}
  Set $p=\max\limits_{J\in [h]^{t}}c_{J},$ the expected number of $(p,J)$-bad sequences is at most $(t(r-1))!N^{t(r-1)}\cdot \frac{O(1)}{q^{m}}=O(q^{t-2}).$ Since the number of types $J\in [h]^{t}$ is $h^{t},$ then Lemma \ref{Lem:BipartiteBadSet} follows by the linearity of expectation.
\end{proof}

Now we are ready to prove the main result of complete bipartite $r$-uniform hypergraphs.

\begin{proof}[\textbf{Proof of Theorem \ref{Thm:bipartite}}]
  Let $\bar{\mathcal{H}}$ be the random multi-hypergraph defined as above. Then by Lemma \ref{Lem:probability}, the expected number of edges in $\bar{\mathcal{H}}$ is $\frac{h}{q}\binom{N}{r}.$ Let $e_{M}$ be the number of multiple edges, we can bound the expected number of $e_{M}$ as
  \begin{equation*}
    \mathbb{E}[e_{M}]\leqslant \binom{N}{r}\sum\limits_{i=2}^{h}\binom{h}{i}q^{-i}=O(N^{r-\frac{2}{t}}).
  \end{equation*}
    As we have shown in Lemma \ref{Lem:BipartiteBadSet}, the expected number of $ph^{t}$-bad sequences is at most $Ch^{t}q^{t-2}.$ We remove all of the multiple edges and remove one vertex from each $ph^{t}$-bad sequence to obtain a new hypergraph $\mathcal{H}',$ since each vertex is contained in at most $O(N^{r-1})$ edges, hence the expected number of edges in $\mathcal{H}'$ is at least
  \begin{equation*}
    \frac{h}{q}\binom{N}{r}-\binom{N}{r}\sum\limits_{i=2}^{h}\binom{h}{i}q^{-i}-O(N^{r-1})Ch^{t}q^{t-2}.
  \end{equation*}
  When $s$ is sufficiently large, let $h=(\frac{s}{p})^{\frac{1}{t}},$ then there exists a $K_{s,t}^{(r)}$-free hypergraph with $\Omega(s^{\frac{1}{t}}n^{r-\frac{1}{t}})$ edges. The proof of Theorem \ref{Thm:bipartite} is finished.
\end{proof}

\subsection{Berge theta hypergraphs}
 Now we define a random hypergraph model that we will use in our construction.

 \begin{definition}\label{Def:BergeRandomgraph}
  For given integers $r$ and $\ell,$ let $d=r\ell^{2}$ and $N=q^{\ell},$ we pick $\ell(r-1)-1$ symmetric polynomials $f_{1},f_{2},\ldots,f_{\ell(r-1)-1}$ from $\mathcal{P}_{d}$ uniformly at random. Let $\mathcal{F}$ be an $r$-partite $r$-uniform hypergraph $\mathcal{F}$ on $rN$ vertices as following: the vertex set $V(\mathcal{F})=\{V_{1},V_{2},\ldots,V_{r}\}$ is $r$ distinct copies of $\mathbb{F}_{q}^{\ell},$ and for $v_{i}\in V_{i},$ $1\leqslant i\leqslant r,$ the $r$-tuple $\{v_{1},v_{2},\ldots,v_{r}\}$ forms an edge of $\mathcal{F}$ if and only if
   $$f_{1}(v_{1},v_{2},\ldots,v_{r})=f_{2}(v_{1},v_{2},\ldots,v_{r})=\cdots =f_{\ell(r-1)-1}(v_{1},v_{2},\ldots,v_{r})=0.$$
 \end{definition}

  We pick $h$ random hypergraphs $\mathcal{F}_{1},\mathcal{F}_{2},\ldots,\mathcal{F}_{h}$ independently and let $\mathcal{\bar{F}}$ be a multi-hypergraph which is the union of the $\mathcal{F}_{i}.$ For the random multi-hypergraph $\bar{\mathcal{F}}$ and positive integer $p,$ we say that a pair of vertices $x,y$ is $p$-bad if there are at least $p$ Berge paths of length at most $\ell$ between $x$ and $y$. Now we need to bound the number of $ph^{\ell}$-bad pairs in $\bar{\mathcal{F}}.$

 \begin{lemma}\label{Lem:BadPairBound}
   Let $B_{ph^{\ell}}$ be the set of all $ph^{\ell}$-bad pairs in $\mathcal{\bar{F}}$, there exist constants $p=p(r,\ell)$ and $C=C(r,\ell)$ such that
   \begin{equation*}
     \mathbb{E}[|B_{ph^{\ell}}|]\leqslant Ch^{\ell}q^{\ell(2-r)}.
   \end{equation*}
 \end{lemma}

 \begin{proof}
   Let $\ell_{0}\leqslant \ell$ be an integer and $K=(k_{1},k_{2}\ldots,k_{\ell_{0}})\in [h]^{\ell_{0}}$ be a fixed type. A Berge path made of edges $e_{1},e_{2},\ldots,e_{\ell_{0}}$ is of type $(k_{1},k_{2},\ldots,k_{\ell_{0}})$ if $e_{j}\in E(\mathcal{F}_{k_{j}})$ for $1\leqslant j\leqslant \ell_{0}.$ For some fixed type $K=(k_{1},k_{2},\ldots,k_{\ell_{0}})\in [h]^{\ell_{0}}$, we say that a pair of vertices $x,y$ is $(p,K)$-bad if there are at least $p$ Berge paths of type $K$ between $x$ and $y$. Since the total number of types is $\sum\limits_{\ell_{0}\leqslant\ell}h^{\ell_{0}}\leqslant \ell h^{\ell},$ we then show that for each fixed type $K,$ there is a constant $p=p(r,\ell)$ such that the expected number of $(\frac{p}{\ell},K)$-bad pairs is $O_{r,\ell}(q^{\ell(2-r)}).$

   The first step is to estimate the expected number of short Berge paths between pairs of vertices. Let $x$ and $y$ be fixed vertices in $\bar{\mathcal{F}}$ and $K=(k_{1},k_{2},\ldots,k_{\ell_{0}})$ be a fixed type. Denote $S_{K}$ as the set of Berge paths of type $K$ between $x$ and $y$. It is difficult to estimate $|S_{K}|$ directly, hence we consider the value of $|S_{K}|^{r\ell},$ which counts the number of ordered collections of $r\ell$ Berge paths of type $K$ from $x$ to $y$. These Berge paths can be overlapping or
identical, and the total number of hyperedges in any collection of $r\ell$ paths is at most $e\ell\ell_{0}.$

      Let $P_{\ell_{0},m}$ be the number of collections of Berge paths between $x$ and $y$ such that their union has $m$ edges in total. Note that the edge set of any particular collection $Y_{m}$ with $m$ edges in $\bar{\mathcal{F}}$ can be written as
      \begin{equation*}
        E(Y_{m})=E(Y_{m,1})\cup E(Y_{m,2})\cup\cdots\cup E(Y_{m,h})
      \end{equation*}
      according to the type $K,$ where $E(Y_{m,i})$ consists of edges from the original hypergraph $\mathcal{F}_{i},$ $i=1,2,\ldots,h.$ Since the random hypergraphs $\mathcal{F}_{1},\mathcal{F}_{2},\ldots,\mathcal{F}_{h}$ are picked independently, by Lemma~\ref{Lem:probability}, we obtain that for $m\leqslant r\ell^{2},$ the probability of any particular collection $Y_{m}$ with $m$ edges is contained in $\bar{\mathcal{F}}$ is $\prod\limits_{i=1}^{h}q^{e(Y_{m,i})(1-\ell(r-1))}=q^{m(1-\ell(r-1))}.$ Hence, we obtain that
   \begin{equation*}
     \mathbb{E}[|S_{K}|^{r\ell}]=\sum\limits_{m=1}^{r\ell^{2}}P_{\ell_{0},m}q^{m(1-\ell(r-1))}.
   \end{equation*}
   We use the argument of He and Tait \cite{TaitBergetheta2018}, which showed that
   \begin{equation*}
     P_{\ell_{0},m}=O_{\ell_{0},r}(q^{m(\ell_{0}(r-1)-1)}).
   \end{equation*}
Hence we obtain that
\begin{equation*}
   \mathbb{E}[|S_{K}|^{r\ell}]=\sum\limits_{m=1}^{r\ell^{2}}P_{\ell_{0},m}q^{m(1-\ell(r-1))}\leqslant\sum\limits_{m=1}^{r\ell^{2}}1=C_{K},
\end{equation*}
where the last inequality holds since $\ell_{0}\leqslant \ell$.

In the next step, we will show that $|S_{K}|$ is either bounded by some constant or is at least $\frac{q}{2}$. However there is no fixed set of polynomials whose set of common roots is exactly $S_{K}$, hence it is difficult to write $S_{K}$ as a variety directly. We need to analyse the set $S_{K}$ as follows.

By the definition of Berge path, each Berge path of length $\ell_{0}$ in $S_{K}$ is a sequence of core vertices and edges such as $(x,e_{1},v_{1},e_{2},\ldots,v_{\ell_{0}-1},e_{\ell_{0}},y).$ We can partition the set of Berge paths into which partite set each core vertex $v_{i}$ is in. Hence for fixed type $K$, $S_{K}$ can be partitioned into disjoint sets depending on which partite set each core vertex belongs to. Denote $S_{t_{1},t_{2},\ldots,t_{\ell_{0}-1}}$ as the set of Berge paths from $x$ to $y$ such that the $i$-th core vertex $v_{i}\in V_{t_{i}}$. If we view $\sigma$ as $\ell_{0}-1$ tuple from $[r]^{\ell_{0}-1}$, then we can write $S_{K}$ as
\begin{equation*}
  S_{K}=\bigcup\limits_{\sigma\in [r]^{\ell_{0}-1}}S_{K,\sigma},
\end{equation*}
and obviously it is a disjoint union.

Fix any $S_{K,\sigma},$ we denote the core vertices in an arbitrary Berge path of length $\ell_{0}$ as $v_{1},v_{2},\ldots,v_{\ell_{0}-1}$ and the non-core vertices in edge $e_{j}$ as $w_{1}^{k_{j}},w_{2}^{k_{j}},\ldots,w_{r-2}^{k_{j}},$ where edge $e_{j}$ is in $\mathcal{F}_{k_{j}}.$ Here we need to make sure that the non-core vertices are ordered based on their partite sets, that is, if $w_{s_{1}}^{i}\in V_{t_{1}}$ and $w_{s_{2}}^{i}\in V_{t_{2}},$ where $s_{1}<s_{2}$, then $t_{1}<t_{2}.$

Now we define the variety $T_{K,\sigma}$ as
\begin{equation*}
  \{z\in \mathbb{F}_{q}^{\ell_{0}(r-1)-1} : f_{i,1}^{k_{1}}(z)=f_{i,2}^{k_{2}}(z)=\cdots=f_{i,\ell_{0}}^{k_{\ell_{0}}}(z)=0, 1\leqslant i\leqslant \ell_{0}(r-1)-1\},
\end{equation*}
where $f_{i,j}^{k_{j}}$ is the $j$-th random polynomial used to define $\mathcal{F}_{k_{j}}$ and $z\in \mathbb{F}_{q}^{\ell_{0}(r-1)-1}$ runs over sequence $(v_{1},\ldots,v_{\ell_{0}-1},w_{1}^{k_{1}},\ldots,w_{r-2}^{k_{1}},\ldots,w_{1}^{k_{\ell_{0}}},\ldots,w_{r-2}^{k_{\ell_{0}}}).$ Note that each $z$ is a vector ordered with the core vertices first and the non-core vertices after.

We then write the polynomials $f_{i,1}^{k_{1}}(z)=f_{i,2}^{k_{2}}(z)=\cdots=f_{i,\ell_{0}}^{k_{\ell_{0}}}(z)$ more accurately as follows.
\begin{flalign*}
  f_{i,1}^{k_{1}}(z) &= f_{i}^{k_{1}}(x,v_{1},w_{1}^{k_{1}},\ldots,w_{r-2}^{k_{1}}), \\
  f_{i,2}^{k_{2}}(z) &= f_{i}^{k_{2}}(v_{1},v_{2},w_{1}^{k_{2}},\ldots,w_{r-2}^{k_{2}}), \\
   & \cdots \\
  f_{i,\ell_{0}}^{k_{\ell_{0}}}(z) &= f_{i}^{k_{\ell_{0}}}(v_{\ell_{0}-1},y,w_{1}^{k_{\ell_{0}}},\ldots,w_{r-2}^{k_{\ell_{0}}}),
\end{flalign*}
where $1\leqslant i\leqslant \ell_{0}(r-1)-1.$ Observe that when $\sigma$ is fixed, the order of all non-core vertices is fixed, hence we also fix the order of arguments given to $f_{i}^{k_{j}}$ according to $\sigma.$ For example, suppose $v_{\ell_{0}-1}\in V_{1}$ and $y\in V_{3}$, then we write $f_{i,\ell_{0}}^{k_{\ell_{0}}}(z)$ as
\begin{equation*}
  f_{i,\ell_{0}}^{k_{\ell_{0}}}(z)=f_{i}^{k_{\ell_{0}}}(v_{\ell_{0}-1},w_{1}^{k_{\ell_{0}}},y,w_{2}^{k_{\ell_{0}}}\ldots,w_{r-2}^{k_{\ell_{0}}}).
\end{equation*}
It is easy to see that $S_{K,\sigma}\subset T_{K,\sigma}.$ However $T_{K,\sigma}$ contains not only all the Berge paths in $S_{K,\sigma}$ but also some degenerate walks which are not Berge paths, hence to obtain $S_{K,\sigma}$ we need to exclude the walks that are not Berge paths. If $T_{K,\sigma}$ contains a degenerate walk $x=v_{0},v_{1},v_{2},\ldots,v_{\ell_{0}-1},v_{\ell_{0}}=y,$ define
\begin{equation*}
  D_{a,b}\triangleq T_{K,\sigma}\cap \{v_{0},\ldots,v_{\ell_{0}},w_{1}^{k_{1}},\ldots, w_{r-2}^{k_{\ell_{0}}}: v_{a}=v_{b}\},
\end{equation*}
for $0\leqslant a<b\leqslant \ell_{0}$ and let $D\triangleq \bigcup\limits_{a,b}D_{a,b}.$ We claim that $D$ is a variety since the union of varieties is a variety, and the complexity of $D$ is bounded.

Now we can use Lemma \ref{Lem:variety} to analyse $S_{K,\sigma}=T_{K,\sigma}\setminus D$. For arbitrary type $K$ and $\sigma\in [r]^{\ell_{0}-1},$ there exists a constant $c(K,\sigma)$ which is dependent on $K$ and $\sigma$, such that either $|S_{K,\sigma}|\leqslant c(K,\sigma)$ or $|S_{K,\sigma}|\geqslant\frac{q}{2}.$ Note that $|S_{K}|=\sum\limits_{\sigma\in [r]^{\ell_{0}-1}}|S_{K,\sigma}|,$ if there exists a $\sigma$ such that $|S_{K,\sigma}|>c(K,\sigma)$, then $|S_{K}|\geqslant|S_{K,\sigma}|\geqslant\frac{q}{2},$ otherwise $|S_{K}|\leqslant c(K,\ell_{0},r)$ for some constant $c(K,\ell_{0},r)$ which is dependent on $\ell_{0}$ and $r.$ With the Markov's inequality, we obtain that
\begin{equation*}
  \mathbb{P}[|S_{K}|> c(K,\ell_{0},r)]=\mathbb{P}[|S_{K}|\geqslant\frac{q}{2}]=\mathbb{P}[|S_{K}|^{r\ell}\geqslant(\frac{q}{2})^{r\ell}]\leqslant \frac{\mathbb{E}(|S_{K}|^{r\ell})}{(\frac{q}{2})^{r\ell}}=\frac{C_{K}}{(\frac{q}{2})^{r\ell}}=O_{\ell_{0},r}(q^{-r\ell}).
\end{equation*}
Let $p\triangleq \ell \max\limits_{\ell_{0}\leqslant\ell}{c(K,\ell_{0},r)},$ then we have the expected number of $(\frac{p}{\ell},K)$-bad pairs is at most $(rN)^{2}\frac{C_{K}}{(\frac{q}{2})^{r\ell}}=O_{r,\ell}(q^{\ell(2-r)}).$ The proof of Lemma \ref{Lem:BadPairBound} is finished by linearity of expectation.
\end{proof}

\begin{proof}[\textbf{Proof of Theorem \ref{Thm:Bergelowerbound}}]
When $\ell\geqslant 2$, let $\bar{\mathcal{F}}$ be a multi-hypergraph defined as above. By Lemma \ref{Lem:BadPairBound}, there are constants $p=p(r,\ell)$ and $C=C(r,\ell)$ such that the expected number of $ph^{\ell}$-bad pairs is at most $Ch^{\ell}q^{\ell(2-r)}.$ Let $\mathcal{F}$ be obtained from $\bar{\mathcal{F}}$ by removing all of the multiple edges.

Since $\mathcal{F}_{1},\mathcal{F}_{2},\ldots,\mathcal{F}_{h}$ are independent random hypergraphs, by Lemma \ref{Lem:probability}, the expected number of edges in $\bar{\mathcal{F}}$ is $hq^{\ell+1}=h(\frac{N}{r})^{1+\frac{1}{\ell}}.$ Let $Y$ be the number of multiple edges, we can bound its expected number as following
\begin{equation*}
  \mathbb{E}[Y]\leqslant N^{r}\sum\limits_{i=2}^{h}\binom{h}{i}(\frac{1}{q^{\ell(r-1)-1}})^{i}=o(N).
\end{equation*}
Then we can remove all $ph^{\ell}$-bad pairs in $\mathcal{F}$ to obtain a new hypergraph $\mathcal{F}'.$ Since each vertex is in at most $O(n^{r-1})$ edges, at most $O(n^{r-1})|B_{ph^{\ell}}|$ edges are removed.
  Therefore the expected number of edges in $\mathcal{F}'$ is at least
  \begin{equation*}
    \mathbb{E}[e(\mathcal{F}')]\geqslant h(\frac{N}{r})^{1+\frac{1}{\ell}}-2N^{r-1}\mathbb{E}[|B_{ph^{\ell}}|]-N^{r}\sum\limits_{i=2}^{h}\binom{h}{i}(\frac{1}{q^{\ell(r-1)-1}})^{i}.
  \end{equation*}
When $t$ is sufficiently large, let $h=(\frac{t}{p})^{\frac{1}{\ell}}$, then there exists a hypergraph $\mathcal{F}'$ which is $\Theta_{\ell,t}^{B}$-free with $\Theta(n)$ vertices and $\Omega_{r,\ell}(t^{\frac{1}{\ell}}n^{1+\frac{1}{\ell}})$ edges.
\end{proof}

\section{Concluding remarks}\label{Sec:remarks}
In this paper, we mainly consider how the specified large parameter of forbidden hypergraph $\mathcal{H}$ affects the Tur\'{a}n number $\textup{ex}_{r}(n,\mathcal{H}).$ Using a variant of random algebraic method, we determine the dependence on such specified large constant for Tur\'{a}n number of complete $r$-partite $r$-uniform hypergraph and complete bipartite $r$-uniform hypergraph. In particular, our results can be reduced to the result of complete bipartite graph, which implies the dependence of  K\H{o}v\'{a}ri--S\'{o}s--Tur\'{a}n's upper bound on large $t$ is correct.

However, we fail in determining whether the upper bound  $\textup{ex}_{r}(n,\Theta_{\ell,t}^{B})=O_{\ell,r}(t^{r-1-\frac{1}{\ell}}n^{1+\frac{1}{\ell}})$ is tight when $t$ is large. We strongly believe this upper bound is tight, for instance, the results of Gerbner, Methuku and Vizer \cite{ImproveBergeK2t} determined the asymptotics for $\textup{ex}_{3}(n,\Theta_{2,t}^{B})=(1+o(1))\frac{1}{6}(t-1)^{\frac{3}{2}}n^{\frac{3}{2}}$. Moreover, for some relatively small $\ell$ and $t$, determining $\textup{ex}_{r}(n,\Theta_{\ell,t}^{B})$ is also of great interest.

\section{Acknowledgement}
The authors express their gratitude to the anonymous reviewers for the detailed and constructive comments which were very helpful for the improvement of the presentation of this paper.

\bibliographystyle{abbrv}
\bibliography{Xzx_bergetheta}
\end{document}